\documentclass[12pt]{amsart}
\usepackage{latexsym, amssymb, amscd}

\theoremstyle{plain}
\newtheorem{Thm}{Theorem}

\newtheorem{Lma}[Thm]{Lemma}
\newtheorem{conj}[Thm]{Conjecture}

\theoremstyle{definition}
\newtheorem{Def}[Thm]{Definition}

\theoremstyle{remark}
\newtheorem{Rem}[Thm]{Remark}

\numberwithin{equation}{Thm}

\newcommand{\brq}{^{[q]}}

\newcommand{\mlabel}[1]%
  {\mbox{}\marginpar{\raggedleft\hspace{0pt}{\rm\ttfamily#1}}\label{#1}}

\newcommand{\into}{\operatorname{\hookrightarrow}}

\newcommand{\e}{\operatorname{e}}

\newcommand{\length}{\operatorname{\lambda}}

\newcommand{\Ann}{\operatorname{Ann}}

\newcommand{\Hom}{\operatorname{Hom}}

\newcommand{\fm}{{\mathfrak m}}
\newcommand{\m}{\fm}
\newcommand{\n}{{\mathfrak n}}

\newcommand{\ringR}{\text{$(R,\fm,K)$ }}

\newcommand{\Dim}{{\rm dim}}

\newcommand{\bx}{\mathbf x}
\newcommand{\bv}{\mathbf v}
\newcommand{\bw}{\mathbf w}
\newcommand{\inc}{\subseteq}
\newcommand{\ehk}{\e_{HK}}

\newcommand{\8}{\infty}

\newcommand{\embdim}{{\rm embdim}}
\newcommand{\ux}{{\underline x}}
\newcommand{\ff}{\textrm{ff}}

\newcounter{hours}\newcounter{minutes}

\newcommand{\excise}[1]{}

\title{New estimates of Hilbert-Kunz multiplicities for local rings of fixed dimension} 

\begin{document}
\numberwithin{Thm}{section}
\numberwithin{equation}{section}

\author{Ian M. Aberbach}

\address{Department of Mathematics, University of Missouri, Columbia, MO 65211}\email{aberbachi@missouri.edu}

\author{Florian Enescu}

\address{Department of Mathematics and Statistics, Georgia State University, Atlanta, GA 30303}\email{fenescu@gsu.edu}
\thanks{The second author was partially supported by the Young Investigator 
Grants H98230-07-1-0034 and H98230-10-1-0166 from the National Security Agency.}

\date{}
\begin{abstract}
We present results on the Watanabe-Yoshida conjecture for the Hilbert-Kunz multiplicity of a local ring of positive characteristic.  By improving on a ``volume estimate'' giving a lower bound for Hilbert-Kunz multiplicity, we obtain the conjecture when the ring either has Hilbert-Samuel multiplicity less than or equal to five, or dimension less than or equal to six. For non-regular rings with fixed dimension, a new lower bound for the Hilbert-Kunz multiplicity is obtained.
\end{abstract}

\maketitle

\section{Introduction}\label{introduction}

Let $\ringR$ be a local ring of positive characteristic $p$. If $I$ is an ideal in $R$, then $I^{[q]}=(i^q: i \in I)$, where
$q=p^e$ is a power of the characteristic. For an
$\fm$-primary ideal $I$, one can consider the Hilbert-Samuel multiplicity and
the Hilbert-Kunz multiplicity of $I$ with respect to $R$. 

\begin{Def}
 {\rm Let $I$ be an $\fm$-primary ideal in $(R,\fm)$.

1. {\it The Hilbert-Samuel multiplicity of $R$ at $I$} is defined by $\e (I)=
\e(I,R) := \displaystyle\lim_{n \to \infty} d! \frac{\length(R/I^n)}{n^d}$. The
limit exists and it is positive.

2. {\it The Hilbert-Kunz multiplicity of $R$ at $I$} is defined by $\e _{HK}
(I)= \e _{HK}(I,R): = \displaystyle\lim_{q \to \infty} 
\frac{\length(R/I^{[q]})}{q^d}$.  Monsky has shown that this limit exists and 
is positive.} \end{Def}

It is known that for parameter ideals $I$, one has $\e(I) = \e_{HK}(I)$. The
following sequence of inequalities is also known to hold: $${\rm max} \{ 1,
\frac{1}{d!} \e (I) \} \leq \e_{HK} (I) \leq \e(I)$$ for every $\fm$-primary
ideal $I$.

We call a local ring $R$ {\it formally unmixed} if  
$\hat{R}$ is equidimensional and ${\rm Min}(\hat{R}) =
{\rm Ass}(\hat{R})$, that is, 
$\dim(\hat{R}/P) = \dim(\hat{R})$ for all its minimal primes $P$, and
all associated primes of $\hat{R}$ are minimal. Nagata calls such
rings {\it unmixed}. 
However, throughout our paper, a local
unmixed ring is a local ring $R$ that is equidimensional and ${\rm Min}(R) =
{\rm Ass}(R)$. 

In this paper we will examine lower bounds for formally unmixed nonregular local rings $R$ of dimension $d$ and prime
characteristic $p$.

\begin{Def}
For $d \geq 1$, let $m_d$ be the real numbers such that

$$ \sec (x) + \tan(x) = 1+ \sum_{d=1}^{\infty} m_d x^d,$$
where $ \mid x \mid < \frac{\pi}{2}$. 

\end{Def}

The following conjecture will be central to our paper:

\begin{conj}[Watanabe-Yoshida, see~\cite{WY3}, Conjecture 4.2 for a more general form]
Let $d \geq 1, p > 2$. Let $K= \overline{\mathbf{F}_p}$ and 

$$R_{p, d} = \dfrac{K[[x_0,\ldots x_d]]}{(x_0^2 + \cdots + x_d^2)},$$

Let $(R, \m, K) $ be a formally unmixed nonregular local ring of dimension $d$.

Then $$ \ehk (R) \geq \ehk (R_{p,d})  \geq1 + m_d.$$

\end{conj}

\begin{Rem}
\label{weak}
The reader should note that the statement $\ehk (R_{p,d})  \geq1 + m_d$ is part of the conjecture. 

This is known for $d \leq 6$, due to Yoshida~\cite{Y}. In fact,
$\ehk(R_{p,5}) = \frac{17p^2+12}{15p^2+10} > m_5=\frac{17}{15}$ and $\ehk(R_{p,6}) = \frac{781p^4+656p^2+315}{720p^4+570p^2+270} > m_6=\frac{781}{720}.$

Therefore the inequality conjectured by Watanabe and Yoshida includes two inequalities: a stronger one 
\begin{equation}
\label{wy1}
 \ehk (R) \geq \ehk (R_{p,d})
\end{equation}
and a weaker one, namely
\begin{equation}
\label{wy2}
 \ehk (R) \geq 1+m_d.
\end{equation}

As far as we know, the inequality $\ehk (R_{p,d})  \geq1 + m_d$ is open for $d \geq 7$.

\end{Rem}

\begin{Rem}
\label{monsky}
Gessel and Monsky have shown (see ~\cite{M}, or Theorem 4.1 in~\cite{WY3}) that 

$$ \lim_{p \to \infty} \ehk (R_{p, d}) = 1 +m_d,$$
for $d \geq 2$.
\end{Rem}

Watanabe and Yoshida have proved this conjecture in dimension $3, 4$. The cases $d=1 ,2$ are also known.

In higher dimensions, it was not known until recently whether or not for a fixed dimension $d$ there exists a lower 
bound, say $C(d) > 1$, such that every local formally unmixed nonregular ring $R$ satisfies $\ehk(R) \geq C(d)$. We have shown
the existence of such lower bound in~\cite{AE}.

\begin{Rem}
If $R$ is a complete intersection of dimension $d \geq 1$ and $p >2 $, then Enescu and Shimomoto (\cite{ES}) have proved that

$$ \ehk(R) \geq \ehk(R_{p,d}).$$

\end{Rem}

In this paper we will develop techniques which will produce improved estimates for Hilbert-Kunz multiplicities of local rings.
In Section 3 we will extend an inequality of Watanabe and Yoshida that gives a lower bound for the Hilbert-Kunz multiplicity
of a local ring $R$ in terms of a volume function. In Section 4 we will apply this inequality to prove the Watanabe-Yoshida conjecture
for rings of Hilbert-Samuel multiplicity at most $5$. Section 5 will provide an asymptotic solution to the above mentioned conjecture
for rings of dimension $5$ and $6$. Furthermore, Section 6 will sharpen the lower bound for the Hilbert-Kunz multiplicity
of a local ring $R$ provided in~\cite{AE} in all dimensions.

Shortly after this paper was posted to the arXiv (arXiv:1101.5078), O. Celikbas, H. Dao, C. Huneke, and Y. Zhang posted a manuscript that obtains a lower bound
of the Hilbert-Kunz multiplicity of a $d$-dimensional ring that improves our bound in certain important cases. Their approach starts with an analysis of radical extensions
like in Section 6 of this paper, however it is a different than ours, and uses along the way new inequalities that are very interesting in their own right.

{\bf Acknowledgement:} The authors are indebted to the anonymous referee for several important suggestions and corrections that improved the manuscript. Notably, the referee observed that Yoshida's results in \cite{Y} can be used to improve our Theorem~\ref{dim5}. Some of the computations in the paper were performed with Wolfram Mathematica~\cite{WM}.

\section{Notations, terminology and background}

First we would like to review some definitions and results that will be useful
later. Throughout the paper $R$ will be a Noetherian ring containing a field of 
characteristic $p$, where $p$ is prime.
Also, $q$ will denote $p^e$, a varying power of $p$.

If $I$ is an ideal in $R$, then $I^{[q]}=(i^q: i \in I)$, where
$q=p^e$ is a power of the characteristic. Let $R^{\circ} = R \setminus
\cup P$, where $P$ runs over the set of all minimal primes of $R$. An
element $x$ is said to belong to the {\it tight closure} of the ideal
$I$ if there exists $c \in R^{\circ}$ such that
$cx^q \in I^{[q]}$ for all sufficiently large $q=p^e$. The
tight closure of $I$ is denoted by $I^\ast$. By a ${\it parameter \
ideal}$ we mean here an ideal generated by a full system of parameters
in a local ring $R$. A tightly closed ideal of $R$ is an ideal $I$
such that $I = I^*$.

Let $F:R \to R$ be the Frobenius homomorphism $F(r)=r^p$. We denote by
$F^e$ the $e$th iteration of $F$, that is $F^e(r) = r^{q}$, $F^e:R
\to R$. One can regard $R$ as an $R$-algebra via the homomorphism
$F^e$. Although as an abelian group it equals $R$, it has a different
scalar multiplication. We will denote this new algebra by $R^{(e)}$.

\begin{Def}
 $R$ is \emph{F-finite}  if $R^{(1)}$ is module finite over $R$, or, 
equivalently (in the case that $R$ is reduced),
 $R^{1/p}$ is module finite over $R$.  $R$ is called 
\emph{F-pure} if  the Frobenius homomorphism is a pure map, i.e,  
 $F \otimes_R M$ is injective for every $R$-module $M$.
\end{Def}

If $R$ is F-finite, then $R^{1/q}$ is module finite over $R$, for
every $q$. Moreover, any quotient and localization of an F-finite
ring is F-finite. Any finitely generated algebra over a perfect
field is F-finite. An F-finite ring is excellent.

\begin{Def}
A reduced Noetherian F-finite ring $R$ is \emph{strongly F-regular}  if 
for every $c \in R^0$ there exists $q$ such that the $R$-linear map $R \to R^{1/q}$ that sends $1$ to $c^{1/q}$ 
splits over $R$, or equivalently $Rc^{1/q} \subset R^{1/q}$ splits over $R$.
\end{Def}

The notion of strong F-regularity localizes well, and all ideals are tightly
closed in strongly F-regular rings. Regular  rings are strongly F-regular
and  strongly F-regular rings are Cohen-Macaulay and normal.

Let $E_R(K)$ denote the injective hull of the residue field of  a local ring $\ringR$. 

\begin{Def} A ring $R$ is called F-rational if all parameter ideals are
tightly closed. A ring $R$ is called weakly F-regular if all ideals are 
tightly
closed. The ring $R$ is F-regular if and only if $S^{-1}R$ is weakly
F-regular for all multiplicative sets $S \subset R$. 
\end{Def}

Regular rings are (strongly) F-regular.  For Gorenstein rings, the notions
of F-rationality and F-regularity coincide (and if in addition the
ring is excellent, these coincide with strong F-regularity).

Our work will rely on a number of inequalities that involve the Hilbert-Kunz multiplicity obtained in~\cite{AE} via 
duality theory, so we will state them here all together.

\begin{Thm}
\label{duality}

Let $(R, \m, K)$ be a local ring of dimension $d$ and characteristic $p$, where $p$ is prime.

(i) Assume that $R$ is Cohen-Macaulay of type $t$. Then

$$ \ehk(R) \geq \dfrac{\e(R)}{\e(R) - t +1}.$$

(ii) Assume that $R$ is Gorenstein of embedding dimension $\nu = \mu(\m)$. If $R$ or $\widehat{R}$  is not
F-regular then

$$ \ehk(R) \geq \dfrac{\e(R)}{\e(R) - \nu +d}.$$

(iii) Assume that $R$ is formally unmixed and $d \geq 2$. 

If $$ \ehk(R) < \dfrac{\e(R)}{\e(R)-1},$$
then $R$ is Gorenstein. Also, $R$ and $\widehat{R}$ are F-regular.

(iv) If $R$ is Cohen-Macaulay and has minimal multiplicity, i.e. $\nu = \e(R) +d -1$, then

$$ \ehk(R) \geq \dfrac{\e(R)}{2}.$$

\end{Thm}

\begin{proof}
Part (i) is Corollary 3.3 in~\cite{AE}. Part (ii) is Corollary 3.7 in~\cite{AE}. Part (iv) is Corollary 3.4 in~\cite{AE}.

For part (iii), by a result of Blickle and Enescu (see for example, Remark 1.3 in~\cite{BE}), we obtain that
$R$ is Cohen-Macaulay. If the type of $R$ is greater than $1$ then part (i) above gives a contradiction.
So, $R$ is Gorenstein and then part (ii) finishes the proof, as $\nu \geq d+1.$

\end{proof}

\section{Volume estimates for Hilbert-Kunz multiplicity lower bounds}

A geometric formula first articulated by Watanabe and Yoshida in \cite{WY3} gives a great deal
of information, especially in small dimension.  We will give an improved
version of their formula here.

For any real number $s$, set
\begin{equation*}
v_s = \textrm{vol}\left\{(x_1,\ldots, x_d) \in [0,1]^d\bigg| \sum_{i=1}^d
x_i \le s\right\}
\end{equation*}

Here ``vol'' denotes the Euclidean volume of a subset of $\mathbb{R}^d$.
In fact, an explicit formula for $v_s$, which is due to P\'olya and can be traced to Laplace (see formula (16) on page 233 in~\cite{CL}\footnote{We thank A. Koldobskiy for providing this reference to us}),  is
\begin{equation*}
v_s = \sum_{n=0}^{\lfloor s \rfloor} (-1)^n \dfrac{(s-n)^d}{n!(d-n)!}
\end{equation*}

\begin{Thm}[c.f., \cite{WY3}, Theorem 2.2]\label{volehk}
Let $(R,\m,K)$ be a formally unmixed local ring of characteristic $p>0$ and
dimension $d$.  Let $J$ be a minimal reduction of $\m$, and let $r$
be an integer with $r \ge \mu_R(\m/J^*)$.  Let $s \ge 1$ be a rational
number.  Then
\begin{equation}
\e_{HK}(R) \ge \e(R)\left\{ v_s - r v_{s-1}\right\}.
\end{equation}
\end{Thm}
%
%
%

Theorem~\ref{volehk} is an improvement over Watanabe and Yoshida's
theorem when the maximum volume occurs for a value of $s > 2$.
Theorem~\ref{volehk} can be made considerably more general.

Fix an
ideal $J$ in an analytically unramified local ring $(R,\m)$.  For an element $x \in R$, set 
$v_J(x) = \sup\{k| x \in J^{k}\}$.
We can then set $f_J(x) = \lim_{n\to \8} \dfrac{v_J(x^n)}{n}$.  By
work of Rees \cite{Re},  the number $f_J(x)$ is rational, and is
the same for any ideal with the same integral closure as $J$.

\begin{Thm}\label{volehk2}
Let $(R,\m,K)$ be a formally unmixed local ring of characteristic
$p >0$ and dimension $d \geq 1$.  Let $J$ be a parameter ideal with $e=e(J)$.  Fix $I \supseteq J^*$ and let $r = \mu_R(I/J^*)$.  Let
$z_1,\ldots, z_r$ be minimal generators of $I$ modulo $J^*$, and
let $t_i = f_J(z_i)$.  For any rational number $s \ge 0$,
\begin{equation}\label{volehk2formula}
\e_{HK}(I) \ge e(v_s - \sum_{i=1}^r v_{s-t_i}).
\end{equation}
\end{Thm}

In order to prove Theorem~\ref{volehk2} we will need Lemma 2.3 of
\cite{WY} (where, for any non-negative real number $\alpha$, we
define $I^\alpha = I^{\lfloor \alpha \rfloor}$):
\begin{Lma}\label{reductionvolume}
Let $(R,\m,K)$ be a formally unmixed local ring of characteristic $p >0$ with
$d = \dim R \ge 1$.  Let $J$ be a parameter ideal of $R$.  Then for
any rational number $s$ with $0 \le s \le d$
\[
\lim_{q\to\8} \dfrac{\length(R/J^{sq})}{q^d} = 
\dfrac {\e(J)s^d}{d!}, \text{\ and } 
\lim_{q\to\8} \dfrac{\length(R/(J^{sq} + J\brq))}{q^d} = \e(J)v_s.
\]
\end{Lma}

Theorem~\ref{volehk} follows from Theorem~\ref{volehk2} by taking
$I = \m$, $J$ a minimal reduction of $\m$, and noting that for
any minimal generator of $\m$, the valuation is at least $1$.

\begin{proof}[Proof of Theorem~\ref{volehk2}]
We can apply Theorem 8.17 (a) in~\cite{HH1} to remark that  $\length((B^*)\brq/B\brq) = O(q^{d-1})$.

Let us note that $I = (z_1, \ldots, z_r) + J^*$. 
  
The proof now follows from an examination of the inequality
\begin{align*}
\length\left(\dfrac R{I\brq}\right) 
            &\ge \length \left(\dfrac R{(z_1, \ldots, z_r)\brq + J^{sq}+ (J^*)\brq}\right) \\
&= \length\left(\dfrac R{(z_1,\ldots, z_r)\brq+ J\brq + J^{sq}}\right) 
   -\length\left(\dfrac{I\brq+ (J^*)\brq + J^{sq}}{I\brq+ J\brq + J^{sq}}\right)
\\&=\length\left(\dfrac R{(z_1,\ldots,z_r)\brq +J\brq + J^{sq}}\right)+O(q^{d-1})\\
&\ge  \length\left(\dfrac R{J^{sq}+J\brq}\right) 
           - \left(\sum_{i=0}^{r-1}\length\left(\dfrac{(z_1,\ldots,z_{i+1})\brq
             +J^{sq}+J\brq}{(z_1,\ldots, z_i)\brq +J^{sq}+J\brq}\right)\right)
+O(q^{d-1})\\
&\ge  \length\left(\dfrac R{J^{sq}+J\brq}\right) -
 \left(\sum_{i=0}^{r-1}\length\left(\dfrac{R}
            {(J^{sq}+J\brq):z_{i+1}^q}\right)\right)
+O(q^{d-1})
\end{align*}

For $N=1,2 , \dots$, let $\epsilon_N = \frac{1}{p^N}$ and choose $q_0 > p^N$ such that for all $q \geq q_0$ we have

$$ \mid \frac{v_J(z_{i+1}^q)}{q} - t_{i+1} \mid < \epsilon_N.$$

Fix $N$. For $q \geq q_0$ we then have $v_J(z_{i+1}^q) \geq \lceil (t_{i+1} -\epsilon_n)q_i\rceil$ and so $z_{i+1}^q \in J ^{ \lceil (t_{i+1} -\epsilon_n)q \rceil} = J^{ \lceil t_{i+1} q  \rceil -\epsilon_n q}$.

It follows that $z_{i+1}^q J^{(s-t_{i+1})q} \subseteq J^{sq-\epsilon_N q}, $ and hence $z_{i+1} ^q J^{sq} \subseteq J^{(s-\epsilon_N +t_{i+1})q}.$

Therefore, 
$$\length\left(\dfrac{R}{(J^{sq}+J\brq):z_{i}^q}\right)  \leq \length\left(\dfrac{R}{(J^{(s-t_{i+1} +\epsilon_N)q}+J\brq)}\right).$$

So,

$$\length\left(\dfrac R{I\brq}\right)  \geq \length\left(\dfrac R{J^{sq}+J\brq}\right) -
 \left(\sum_{i=0}^{r-1}\length\left(\dfrac{R}
            {(J^{(s-t_{i+1}+\epsilon_N)q}+J\brq}\right)\right)
+O(q^{d-1}).$$

Dividing each term in the last inequality obtained by $q^d$, taking limits as $q\to \8$,
and applying Lemma~\ref{reductionvolume} to each term plus the fact that $\lim_{\epsilon \to 0} v_{s-\epsilon} = v_s$ yields equation~\ref
{volehk2formula}.
\end{proof}

\begin{Rem}
This result also extends Fact 2.4 in~\cite{WY3}.
\end{Rem}

\section{Lower bounds for rings with small Hilbert-Samuel multiplicity}

In this section we will apply Theorem~\ref{volehk2} to provide lower bounds for the Hilbert-Kunz multiplicity of formally 
unmixed local ring of Hilbert-Samuel multiplicity less or equal to $5$. 

We note that $$1+m_3=\frac{4}{3}, 1+m_4 = \frac{29}{24}, 1+m_5 = \frac{17}{15}, 1+m_6 = \frac{781}{720} = 1.0847.$$

\begin{Thm}
\label{mult=3}
Let $(R,\m,K)$ be a Cohen-Macaulay local ring such that $\e(R) = 3$ and $R$ is
not a complete intersection.  Then $\e_{HK}(R) \ge 13/8$.
\end{Thm}

\begin{proof}
We may immediately complete.
Let $d = \dim R$ and $k = \embdim(R) - \Dim(R)$. It is known that $k \leq \e-1 =2$.  
Since $R$ is not a complete intersection then $k >1$, so $R$ is 
a ring of minimal multiplicity.  Sally's Thereom 1.1 in~\cite{Sa}, gives that we can write $R = S/I$  where
$S = K[[x_1,\ldots, x_{d+2}]]$. The same result implies that $I$ is a 3-generated ideal of $R$ and
that  the Hilbert-Burch theorem applies, so
$I$ is the ideal of minors of a $3 \times 2$ matrix, say
$\bmatrix a_{ij}\endbmatrix$, where
$a_{ij} \in (x_1,\ldots, x_{d+2})S$.  

Consider the ring 
$R_1 = K[[y_{11},\ldots, y_{32},x_1,\ldots, x_{d+2}]]/I_2([y_{ij}])$.
Then $\dim R_1 = 4 + d + 2 = d+6$.  

Clearly, $R_1/(y_{ij}-a_{ij} |
1\le i\le 3, 1 \le j\le 2) \cong R$.  Since $\dim R_1 - \dim R = 6$,
the equations form a regular sequence, so $\e_{HK}(R) \ge \e_{HK}(R_1)$,
and $$\e_{HK}(R_1) = \e_{HK}\big(K[[y_{11},\ldots, y_{32}]]/I_2([y_{ij}])\big)
= 13/8,$$ (the ring $R_1$ is isomorphic to the Segree product $S_{2,3}$ and so Theorem 3.3 in~\cite{EY} gives the value $\frac{13}{8}$).
\end{proof}

\subsection*{Case of a local ring of Hilbert-Samuel multiplicity $3$:}

\medskip
\noindent
Let $(R, \fm)$ be a formally unmixed local ring of multiplicity $\e =3$ and characteristic $p >2$. 
We can complete and assume that $R$ is complete and unmixed.

If $\ehk(R) < \frac{\e}{\e-1} = 1.5$ we have that $R$ is Gorenstein, by Theorem~\ref{duality} (iii). In this case, by 
Theorem~\ref{mult=3} if $R$ is not a complete
intersection then $\ehk(R) \geq \frac{13}{8}$. Otherwise $\ehk(R) \geq \ehk(R_{p,d})$ by Enescu-Shimomoto. This shows that the
Watanabe-Yoshida conjecture is settled for local rings of multiplicity $3$.

\subsection*{Case of a local ring of Hilbert-Samuel multiplicity $4$:}

\medskip
\noindent
Let $(R, \fm)$ be a formally unmixed local ring of multiplicity $\e =4$ and characteristic $p >2$. 
We can complete and assume that $R$ is complete and unmixed. Let $k =\embdim(R) -\dim(R)$. 

If $\e_{HK}(R) < 1 + 1/(4-1) = 4/3$, then $R$ is Gorenstein by Theorem~\ref{duality} (iii).  Since $k \leq \e-1 =3$, then if
$R$ has minimal multiplicity ($k=3$), then $\e_{HK}(R) \ge 4/2=2$ by Theorem~\ref{duality}(iv). If  $k=2$, by considering the minimal free resolution of $R$ over $S$, we see that $R$ is a complete intersection. The case $k=1$ also leads to $R$ being a complete intersection.
In both cases $\ehk(R) \geq \ehk(R_{p,d})$ by Enescu-Shimomoto. This shows that the Watanabe-Yoshida conjecture is settled for local rings of multiplicity $4$.

\subsection*{Case of a local ring of Hilbert-Samuel multiplicity $5$:}

\medskip
\noindent
Let $(R, \fm)$ be a formally unmixed local ring of multiplicity $\e =5$ and characteristic $p >2$. 
We can complete and assume that $R$ is complete and unmixed. Let $d = \Dim(R)$.

We can assume that $R$ is Gorenstein if $\e_{HK} < 1.25$
by Theorem~\ref{duality} (iii).

Let us assume that $R$ is Gorenstein and set $k =\embdim(R) -
\Dim(R)$. If $ k = \e-1$ then $R$ has minimal multiplicity and then
Theorem~\ref{duality} (iv) gives $\ehk(R) \geq \e / 2 = 2.5$. So we
can assume that $k \leq \e-2 =3$. In fact, the cases $k =1,2$ both
imply that $R$ is complete intersection (the case $k=2$ follows from
Serre Theorem as in~\cite{Sa3} Theorem 1.2 page 69).

If $k=3$ then write $R$ as $S/I$ where $S = K[[x_1,\dots,
x_{d+3}]]$ is complete local regular and $I$ is a height 3 Gorenstein
ideal with $I \subset \n^2$ where $\n = (x_1, \ldots, x_{d+3})$. By the
Buchsbaum-Eisenbud Structure Theorem (see Theorem 1.5 page 72
in~\cite{Sa3}) the ideal $I$ is given
by the set of Pfaffians of a $5\times 5$ anti-symmetric
matrix with entries in $S$. The upper right corner  has at most 10 non-zero
entries denoted  $a_{ij}, 1 \leq i < j \leq 5$. These elements belong
to $\n$.

Let $A = (y_{ij})$ be an antisymmetric matrix of indeterminates of size $5 \times 5$ and
set

$$R_1 = K[[y_{ij}, x_1, \ldots, x_{d+3}: 1 \leq i < j \leq 5]]
/((Pf(A)),$$
where $(Pf(A))$ is the ideal generated by the Pfaffians of $A$.

We note that $\dim(R_1) = 7 + d +3 = 10 +d.$ Also, the elements $y_{ij} - a_{ij}, 1
\leq i < j \leq 5$ form a regular sequence in $R_1$ since 
$R_1/( y_{ij} - a_{ij}, 1
\leq i < j \leq 5) \simeq R,$ 
and the dimension drops exactly by $10$.

Therefore

$$\ehk(R) \geq \ehk(R_1) = \ehk (K[[y_{ij}: 1 \leq i < j \leq 5]]
/((Pf(A)),$$ 
and the former is a Gorenstein ring of dimension $7$ and multiplicity
$5$.

So, it remains to examine $7$-dimensional Gorenstein rings of
multiplicity $5$.

Let $J$ be an ideal generated by an s.o.p.. Since $\mu (\fm) = d+3$ and $d=\dim(R)$, we get that $3 \geq \mu(\fm/J) \geq \mu (\fm/ J^*)$. 

Using the notations from Theorem~\ref{volehk2}, we note that $\e (v_s - \mu(\fm/J^*)v_{s-1}) \geq \e(v_s - 3v_{s-1})$.

Now apply Theorem~\ref{volehk2} with $\e=5$ and $s = 3.32$ and  get $\e_{HK}(R) \ge 1.112$ (we used Mathematica to compute the
volume functions).

\section{Watanabe-Yoshida Conjecture for rings of dimension $5$ and $6$}

In this section we will show how to use Theorem~\ref{volehk2} to prove the Watanabe-Yoshida conjecture in 
dimensions $5$ and $6$ for large enough $p$. 

We note that $$m_5 = \frac{17}{15}, \quad m_6 = \frac{781}{720} = 1.0847.$$ We need results of Goto and Nakamura~\cite{GN}, 
Theorems 1.1 and 1.2.

\begin{Thm}
\label{gn}
Let $(R, \m, K)$ be a homomorphic image of a Cohen-Macaulay ring. Assume that $R$ is equidimensional.

Then for every parameter ideal $I$ in $R$ we have 

$$ \e(I) \geq \length(R/ I^*).$$
In fact, under the assumption that $R$ is a homomorphic image of a Cohen-Macaulay ring and 
${\rm Ass}(R) = {\rm Assh}(R),$ if

$$ \e(I) = \length(R/ I^*),$$ for some parameter ideal $I$, then $R$ is Cohen-Macaulay and $F$-rational.

\end{Thm}

We can prove the following

\begin{Thm}
\label{dim5}

Let $(R, \m, K)$ be a formally unmixed local nonregular ring of dimension $d$ and positive prime characteristic $p>2$. Then

(i) If $d=5$, then $$\ehk (R) \geq \ehk(R_{p,d}) \geq \frac{17}{15} =1 + m_5$$

(ii) If $d=6$, then $$\ehk(R) \geq \ehk(R_{p,d}) \geq \frac{781}{720} = 1+m_6.$$

\end{Thm}

\begin{proof}
We can complete $R$ and enlarge the residue field of $R$ so that it is infinite. The associativity formula for the 
Hilbert-Kunz multiplicity shows that for an unmixed ring $R$
$\ehk(R) < 2$ implies that $R$ is domain (as in Remark 2.6 in~\cite{AE}). Therefore, we can assume that $R$ is domain.

Let $\underline{x}$ be a minimal reduction for $\m$. Set $J = (\underline{x})$. Note that we are in the case of $R$ complete and domain.
Set $\e= \e(R)$.

We claim that either $R$ has minimal multiplicity or $ \mu(m/J^*) \leq \e -2$.

If $R$ is not $F$-rational then $\e(J)  > \length(R/ J^*)$. So, $\e =\e(J)  >  1+ \length(\m/J^*) \geq
1+ \mu (\m/J^*)$. In other words, $\e-1 > \mu (\m/J^*)$ or $ \e-2 \geq \mu( \m/J^*)$.

Now let us assume that $R$ is Cohen-Macaulay and F-rational. Then $\e =\e(J) =\length(R/ J) = \length (R/J^*)$. In conclusion,
$\length(\m/ J^*) = \e -1$. Since $\mu ( \m/ J^*) \leq \length(\m/J^*) \leq \e-1$, we see that $ \mu(m/J^*) > \e -2$ is only possible when
$\mu (\m/ J^* ) = \length (\m/ J^*)$. Recall that $J^* =J$. So we get $\mu (\m/ J ) = \length (\m/ J)$. But,
$\mu (\m/ J)  = \dim (\m /\m^2 +J) = \length (\m /\m^2 +J)$.  Hence $\mu (\m/ J ) = \length (\m/ J)$ 
leads to $\m^2 \subseteq J$. But it is well-known that $\m^2 \subseteq J$ implies $\m^2 =\m J$. This proves that $R$ is of minimal multiplicity by Theorem 3.8 page 45 in ~\cite{Sa3}.

Our claim is now proved. In the minimal multiplicity case the Theorem~\ref{duality} (iii) implies that 
$\ehk(R) \geq 1.5 \geq \ehk{(R_{p,d})}$, by Remark~\ref{weak}, or $\e =2$ in which case $R$ is a 
hypersurface and then $\ehk (R) \geq \ehk(R_{p, d})$ by Enescu-Shimomoto. 

Hence, in the minimal multiplicity case, the Watanabe-Yoshida conjecture is true.

So we have reduced our analysis to the case $ \mu(m/J^*) \leq \e -2$. Let $r = \mu(m/J^*).$

Theorem~\ref{volehk2} implies that

$$ \ehk(R) \geq \e \cdot (v_s -rv_{s-1}) \geq \e \cdot (v_s -(\e-2)v_{s-1}).$$

In fact if $\e \geq \e_0$  and $r_0 \geq \e-2$ then also

\begin{equation}
\label{ineq}
\ehk(R) \geq \e_0 \cdot (v_s -r_0v_{s-1})
\end{equation}

Let us consider the case $d=5$.

Let $\e =\e(R).$ If $\e \geq 137$, then $\ehk (R) \geq \e(R) / d!$ implies that $\ehk (R) \geq 137/5! = 137/120 = 1.141(6)$.

Let us assume now that $\e \leq 136$. We will apply inequality (3.1) repeatedly by giving values to $\e_0, r_0$, and $s$.

In the table below we list these choices together with the corresponding lower bound obtained for $\ehk(R)$.

\medskip

\begin{center}
\begin{tabular}{ c  |  c  |  c  | c  ||  c }
$\e$ & $\e_0$ & $r_0$ & s & $\ehk$ \\ \hline
$35 \leq \e \leq  136$ & 35 &134 & 1.4& $\geq 1.153$ \\
$18 \leq \e \leq 34$ & 18 & 32 & 1.7 & $\geq 1.197$ \\
$11 \leq \e \leq 17$ & 11 & 15 & 1.9 & $\geq 1.187$ \\
$7 \leq \e \leq 10$ & 7 &  8& 2.1 & $\geq 1.161$ \\
$5 \leq \e \leq 6$ & 5 & 4 & 2.4 & $\geq 1.313$ \\
\end{tabular}
\end{center}

Now, let us move to the case $d=6$.

Again, we may assume that $\e \geq 5$. For $\e \geq 786$ we obtain $\ehk \geq 786/6! = 786/720$.

We will now show that $$ G(e) : = \e(v_s-(\e-2)v_{s-1}) \geq \frac{786}{720}$$ for all $5 \leq \e \leq 785$.

Since $G(e) = -v_{s-1}\e^2+(v_2+ 2 v_{s-1})e$ is a quadratic function in $\e$, we conclude that for a fixed $s$, the maximum value of $G$ is attained
at $\e = m : = \frac{v_s+2v_{s-1}}{2v_{s-1}}$.

This implies that for $a \leq m \leq b$ 

\begin{equation}
\label{ineq2}
G(e) \leq \min(G(a), G(b))
\end{equation}

The formula for $v_s$ gives the following: $v_s = \frac{s^6}{6!}$, for $0\leq s < 1$; $v_s = \frac{s^6}{6!} - \frac{(s-1)^6}{5!}$, for $1 \leq s <2$ and $v_s = \frac{s^6}{6!} - \frac{(s-1)^6}{5!} -\frac{(s-2)^6}{2 \cdot 4!}$, for $2 \leq s <3$.

For $1 \leq s < 2$, we obtain $m = \dfrac{s^6-4(s-1)^6}{2 (s-1)^6}$. For $2 \leq s < 3$, we obtain $m =  \dfrac{s^6-4(s-1)^6 +3 (s-2)^6}{2 (s-1)^6 - 6(s-2)^6}$.

If $296 \leq \e \leq 786$, then by letting $s=1.3$ we obtain $m \geq 3308.57 > 786$. This gives that $G$ is increasing on $[286, 786]$ which shows that on this interval $G(e) \geq G(296) > 1.89$ and so 
$\ehk \geq 1.89$.

For the rest of the analysis, as in the paragraph above, we will consider intervals $[a,b]$ containing $\e$, give a specific value to $s$ and then compute the resulting value for $m$. In each case, $m$ will happen to land in $[a,b]$ and hence inequality~\ref{ineq2} will apply.

The numbers including those for specific values for $G$ are computed using Mathematica and we usually present our numbers while keeping the first decimal point only.

\medskip

\begin{center}
\begin{tabular}{c | c | c | c | c}
$[a, b]$ & $s$ & $m$ & $\min(G(a), G(b))$ & $\ehk \geq $ \\ \hline
$[59, 296]$ & 1.6 & 177.7 & $G(59)$  & 1.133 \\
$[26, 58] $ & 1.9 & 42.2 & $G(26)$ & 1.123 \\
$16, 25]$ & 2.1 & 22.2 & $G(16)$ & 1.118 \\
$[10, 25]$ & 2.2 & 13.3& $G(10)$ & 1.118\\
$[5, 9]$ & 2.6 & 7.3 & $G(5)$ & 1.107 \\

\end{tabular}
\end{center}

\end{proof}

\section{Root extensions and comparison of Hilbert-Kunz multiplicities}
\label{root-compare}

The next theorem we prove allows us to use Theorem~\ref{volehk2} to obtain
lower bounds for Hilbert-Kunz multiplicities that are not available
using Theorem~\ref{volehk}.

We will need to use a result of Watanabe and Yoshida (\cite{WY}, Theorem
2.7). Let $ff(A)$ denote the total ring of fractions of a ring $A$.

\begin{Thm}\label{extendehk}
Let $(R,\m) \into (S,\n)$ be a module-finite extension of local domains.
Then for every $\m$-primary ideal $I$ of $R$, we have
\begin{equation}
e_{HK}(I) = \dfrac{e_{HK}(IS)}{[ff(S):ff(R)]} \cdot [S/\n:R/\m].
\end{equation}
\end{Thm}

\begin{Def}
Let $(R,\m)$ be a domain. Let $z \in \m$ and $n$ a positive integer. Let $v \in R^+$ be any 
root of $f(X)=X^n-z$.  We call $S = R[v]$ a {\it radical extension} for the pair $R$, $z$.
\end{Def}

It should be remarked that whenever $S$ is radical for $R, z$, then $b:=[ff(S):ff(R)] \leq n$. In what follows $\n$ will denote the maximal ideal of $S$.

\begin{Lma}
\label{bound}
Let $(R, \m, K)$ be a domain and $(S=R[v], \n)$ a radical extension
for $R$ and $z\in R$. Assume that $K$ is algebraically closed. 
Let $I \inc R$ be such that $z\notin I$ and $\m = (z)+I$.  Suppose that
$J = (zr)+I_0 \inc R$ is an ideal such that $\length_R(J/I_0) = 1$ and
in $S$, $vrIS \inc I_0 S$ (one such possibility is $J =\m = (z)+I$).
 Let $b = [ff(S) : ff(R)]$.

Then $$\e_{HK}(I_0, J) \leq \frac{n}{n-1} \ehk(R) - \frac{n}{b(n-1)}\ehk(S).$$

\end{Lma}

\begin{proof}
Consider the following sequence of inclusions:

$$ \m S \subset (\m, v^{n-1})S \subset \cdots \subset (\m, v^2)S 
\subset (\m,v)S =\n.$$


It is easy to see that 

$$(\m, v^j)\brq S: v^{q(j-1)} \subset (\m, v^{j+1})\brq S : v^{qj},$$

since, if $cv^{q(j-1)} \in (\m, v^j)\brq S$, then $c v^{qj} \in  (\m,
v^j)\brq v^q S \subset (\m , v^{j+1})\brq S$.

Thus $\ehk(\m S, \n) = \sum_{j=1}^{n-1} \ehk((\m, v^{j+1})S, (\m, v^{j} )S) 
\geq (n-1) \ehk(\m S, (\m, v^{n-1})S ).$

Consider now the filtration

$$ 
I_0 S \inc (I_0, zrv^{n-1})S \inc \cdots \inc (I_0, zrv)S \inc (I_0,zr)S =JS.
$$

Let $s \in (\m\brq S:_S v^{(n-1)q}) = (v^n,I)\brq S:v^{(n-1)q} 
= v^qS+ I\brq S:_S v^{(n-1)q}$.  Then for any $0 \le j < n$,
\begin{align*}
s(zrv^j)^q \in (v^qS+ I\brq S:_S v^{(n-1)q})(zrv^j)^q 
&\inc (zrv^{(j+1)})^qS + (I\brq S:_S v^{(n-1)q})(v^{(n-1)q}r^q v^{(j+1)q}) \\
&\inc (zrv^{(j+1)})^q)S + I\brq r^q v^{(j+1)q}S \inc (zrv^{(j+1)}, I_0)\brq S.
\end{align*}

Thus $\ehk((I_0,zrv^{j+1})S, (I_0,zrv^{j})S) \le \ehk(\m S, (\m, v^{n-1})S )$.

%
%
%

Since in the chain we have at most $n$ inclusions we get,
using Theorem \ref{extendehk} that
$b \ehk(I_0,J) = \e_{HK}(I_0S, (I_0,zr)S) \leq 
n \ehk( \m S, (\m , v^{n-1})S) \leq \frac{n}{n-1} \ehk (\m S, \n) 
= \frac{n}{n-1} (b \ehk(R) - \ehk(S)) $, 

which gives 

$$\ehk (I_0, J) \leq \frac{n}{n-1}\ehk(R) - \frac{n}{b(n-1)} \ehk(S).$$
\end{proof}

\medskip

\noindent
In what follows we consider a Gorenstein local  domain
$(R,\m, K)$ with algebraically closed residue field. Let us fix some notation. Let $d = \dim (R)$ and consider a system of
parameters $\underline{x} = x_1, \cdots, x_d$ that generates a minimal
reduction of $\m$. Also, $k = {\rm embdim}(R) - \dim(R)$. We plan to provide a lower bound greater than $1$ for the Hilbert-Kunz 
multiplicity of $R$. We also assume that $p \neq 2$.  Note that if $k =2$ and $R$ is Gorenstein, then $R$ is a complete intersection.  This is because, after completing, $R$ is the quotient of a regular ring of dimension $d+2$ and has projective dimension 2 over the regular ring.  The only possible resolution in this case is of a regular sequence over the regular ring.

The main result in~\cite{ES} gives the conjectured lower bound for $\ehk(R)$ if $R$ is a complete intersection. So, we will assume that $R$ is not a
complete intersection, hence $k \geq 3$. Moreover, by a result of J.~Sally (Corollary 3.2 in~\cite{Sa2}), no Gorenstein rings except
hypersurfaces can have minimal multiplicity (i.e., $\e(R) = \mu(\m)-d+1$)
 so $\e = \e(R) \geq
k+2$. In particular, $\e \geq 5$.

\begin{Lma}
\label{prep}

Let $(R, \m, K)$ be a local Gorenstein ring, $k = {\rm embdim}(R) -
\dim(R)$ and $\e =\e(R)$. Let $\underline{x} =x_1, \ldots, x_d$ be a
system of parameters for
$R$.

i) The $R/(\underline{x})$-module $\dfrac {(\underline x):\m^2}{(\underline x)}$ is $k$-generated with
one dimensional socle.

ii) Assume that $\underline{x}$ is a minimal reduction for $\m$. Then $ k = \e-2$ if and only if $(\underline{x}) : \m^2 = \m$.

\end{Lma}

\begin{proof}
For i), note that $R/(\underline{x})$ is Gorenstein and hence we can use Matlis duality. The module $\dfrac {(\underline x):\m^2}{(\underline x)}$ is Matlis dual
to $\dfrac R{(\underline x)+\m^2}$.  $\dfrac R{(\underline x)+\m^2}$
is cyclic with $k$ dimensional socle, therefore 
$\dfrac {(\underline x):\m^2}{(\underline x)}$ is $k$-generated with
one dimensional socle.

To prove part (ii), we recall Proposition 4.2 in~\cite{Sa2} which says in our case that
$k =\e-2 $ if and only if $\m^3 \subset (\underline{x}) \cdot \m$ and $\length(\m^2 / (\underline{x}) \cdot \m)=1$. Hence one direction of (ii) follows at once.
Now assume that $\m ^3 \subset (\underline{x})$. Note that 
$$(\underline{x}) \subseteq \m (\underline{x} : \m^2)+ (\underline{x}) \subseteq (\underline{x}) : \m \subset (\underline{x} : \m^2),$$
and since $R$ is Gorenstein we must have $\m (\underline{x} : \m^2)+
(\underline{x}) = (\underline{x}) : \m$.

Therefore $ \m \cdot \dfrac {(\underline
  x):\m^2}{(\underline{x})} = \dfrac{(\underline{x}) :
  \m}{(\underline{x})},$ and this shows that

$$ k = \mu (\dfrac {(\underline{x}):\m^2}{(\underline x)}) = \dim _K
(\dfrac{(\underline{x}) : \m^2}{(\underline{x}): \m}) = \dim_K
  (\dfrac{\m}{(\underline{x}):\m}) = \length
  (\dfrac{\m}{(\underline{x}):\m}) = \e-2,$$
because $\length(R/\underline{x})  = \e$ ($\underline{x}$ forms a
minimal reduction for $\m$.)

\end{proof}

Let $(R, \fm, K)$ be a local ring with infinite residue field and of dimension $d$. According to a result due to Northcott and Rees and, independently, Trung (see Theorem 8.6.6 in~\cite{HS}) there exists a Zariski open subset $U$ of $(\fm /\fm ^2)^d$ such that any $x_1, \ldots, x_d$
with $(x_1+ \fm^2, \ldots, x_d + \fm^2) \in U$ form a minimal reduction for $\fm$. We will call a set $U$ with this property {\it reduction open}.

\begin{Lma}
\label{choice}

Let $(R, \m, K)$ be a local Gorenstein ring containing an infinite field of positive prime characteristic $p >2$. Assume that $k={\rm embdim}(R) -\dim(R) \geq 2$.
Let $U$ be a reduction open subset of $(\fm/\fm^2)^d$. Let $\underline{x}$ be in $\m$ such that  $(x_1+ \fm^2, \ldots, x_d + \fm^2) \in U$

Then, we may pick minimal
generators $z_1, \ldots, z_k$ for $\dfrac {(\underline{x}):\m^2}
{(\underline x)}$ and a minimal generator $z$ of $\m$ such that $z z_i \notin (\underline{x})$ for $1 \le i \le k$ and 
$z, x_2, \ldots, x_d$ form a minimal reduction of $\m$.

If $k \neq \e-2$, then $z$ can be picked not in $(\underline x):\m^2$. If $k=\e-2$, one may take $z=z_1$.

\end{Lma}

\begin{proof}

Clearly $$(\underline{x}) \subseteq \m (\underline{x} : \m^2)+ (\underline{x}) \subseteq (\underline{x}) : \m \subset (\underline{x} : \m^2),$$
and since $R$ is Gorenstein we must have $\m (\underline{x} : \m^2)+ (\underline{x}) = (\underline{x}) : \m$. This is the case because
$(\underline{x}) = \m (\underline{x} : \m^2)+ (\underline{x})$ gives $\m (\underline{x} : \m^2) \subseteq (\underline{x})$ or
$(\underline{x} : \m^2)= (\underline{x}) : \fm$ which contradicts the fact that $k \geq 2$.

Choose $z_1, \ldots, z_k$ in $R$ such that their images form a minimal set of generators for $\dfrac {(\underline{x}):\m^2}
{(\underline x)}$. We conclude that each $z_i \notin  \m (\underline{x} : \m^2)+ (\underline{x})$, and so
$z_i \notin (\underline{x}) : \fm$, $i=1, \ldots, k$. Note that $z_i \in (\underline{x}) : \fm^2$ and hence $\fm^2 \subset (\underline{x}): z_i$ for all $i=1, \ldots, k$.

Let $U_1 = \{ z + \fm ^2 \in \fm/\fm^2 : (z+\fm^2, x_2 +\fm^2, \ldots, x_d+\fm^2) \in U \}$. Then $U_1$ is  a Zariski open subset  of $\fm / \fm^2$. In what follows, for $a\in R$, $\overline{a}$ will denote
the class of the element $a \in R$ modulo $\fm^2$, $\hat{a}$ the class of $a$ in $R/\fm$ and $\tilde{a}$ the class of $a$ in $R/(\underline x)$.

Then $$ U_1 \not\subseteq \cup_i ((\underline{x}) : z_i)/\fm^2,$$ since otherwise there exists $i$ such that
$U_1 \subseteq ((\underline{x}) : z_i)/\fm ^2$ which gives $\fm \subseteq ((\underline{x}) : z_i)$ or $z_i \in (\underline{x}) : \m$ which is not the case (over an infinite field, a dense Zariski open 
subset cannot be covered by a finite union of proper vector subspaces due to dimension reasons.)

Note that $(\underline{x}: \fm^2) + \fm^2 = \fm$ implies by NAK that $(\underline{x}: \fm^2) =\fm$. So,  a similar argument shows that when $\m \neq (\underline{x} : \m^2)$ one has that 

$$ U_1 \not\subseteq \cup_i ((\underline{x}) : z_i)/ \m^2 \cup ((\underline{x}: \fm^2) + \fm^2)/\fm^2.$$

This guarantees that, in either case, one can pick $z$ a minimal generator of $\m$ such that 
$z z_i \notin (\underline{x})$ for $1 \le i \le k$ and that $z \notin (\underline x):\m^2$, whenever
$\m \neq (\underline{x} : \m^2)$. 

Let us note that $k =\e-2$ is equivalent to $(\underline{x} : \m^2) = 
\m$ by Lemma~\ref{prep}.

Whenever  $(\underline{x}  : \m^2) = \m$, we know that no $z_i$ can kill all $z_j$ modulo $\underline{x}$.
So for all $i,j$, there exists $r_{ij} \in R$ such that $\tilde{z_i}\tilde{z_j} = r_{ij} \tilde{u}$, where $u$ gives the socle generator of $R/(\underline{x})$. Here,
each $r_{ij}$ is an element in $R$, and for each $i$ there exists $j$ such that $\hat{r}_{ij}$ in $R/\fm$ is  nonzero. After renumbering, we can assume that $\hat{r}_{12} \neq 0$. Since $R$ contains an infinite field we have that $K$ is infinite as well.
Let $z'_1 = z_1 +yz_2$ where $y\in R$.  Let $C$  be the set $ \{\overline{z_1} +\overline{y} \cdot \overline{z_2}, y \in R \}$ in 
$\m / \m^2$. This is a line in the $(k+d)$-dimensional space $\m / \m^2$.

Let $z'_j = z_j +yz'_1 = z_j + yz_1+y^2z_2$ for all $j \geq 2$.

We will find $y \in R $ such that
${z'_1}^2 \notin (\underline{x}) $ and for all $j \geq 2$, $z'_1 z'_j \notin (\underline{x})$, and $\overline{z'_1} \in U_1 $.

Computing $\tilde{z'_1}^2 = (\hat{r}_{11}^2 + 2\hat{r}_{12}\hat{y}  + \hat{r}_{22}\hat{y}^2)\tilde{u}$ and $\tilde{z'_1}\tilde{z_j} = [\hat{r}_{1j} + (\hat{r}_{2j}+ \hat{r}_{11})\hat{y} +2\hat{r}_{12}\hat{y}^2 + \hat{r}_{22}\hat{y}^3]\tilde{u}$, $j =2,\ldots, k$ gives $k$ polynomial functions in $\hat{y} \in K$.  Each polynomial is not identically zero because
$2 \hat{r}_{12} \neq 0$. Let $U = \{ \hat{y} \in R/\fm =K : \tilde{z'_1}^2 \neq 0, \tilde{z'_1}\tilde{z'_j} \neq 0, \forall j =2, \ldots, k \}$.
This is an open nonempty subset of $K$. For any choice of $y \in R$ such that $\hat{y} \in U$ we have that ${z'_1}^2 \notin (\underline{x}) $ and for all $j \geq 2$, $z'_1 z'_j \notin (\underline{x})$.

Note that $C \cap U_1$ is an open subset in $C$. Since $C$ is isomorphic to  $K$ we have a open subset of $K$, say $U'$, such that  for all $y \in R$ such that $\hat{y} \in U'$, $\overline{z}_1+ \overline{y} \cdot \overline{z}_2$ belongs to $U_1$.
Now, since $K$ is infinite $U'$ and $U$ must intersect so we can choose $y \in R$ such that $\hat{y} \in U \cap U'$.

To finish the argument here, it is enough to note that we can swap now $z_1$ for $z'_1$ and $z'_j$ for $z_j$ corresponding to our choice for $y$, and the conditions are now satisfied.

\excise{
Let $t$ equal the maximum number of generators $z_1, ..., z_t$ of $(\underline{x} : \m^2)$ modulo $(\underline{x})$ such that
there exist a minimal generator $z$ of $\m$ such that $zz_i \notin (\underline{x})$. 

 Note that
$t \geq 1$. We will show that $t=k$. Assume $t < k$. Complete the list to a full set of minimal generators $z_1, ..., z_t, z_{t+1}, ..., z_k$ of 
$(\underline{x} : \m^2)$ modulo $(\underline{x})$. For $ i \geq t+1$ we must have $zz_i \in (\underline{x})$.

Note that each $z_i$ belongs to $(\underline{x} : \m^2)$ but is not in $\m (\underline{x} : \m^2)+ (\underline{x})$.
Clearly $$(\underline{x}) : \m \subseteq \m (\underline{x} : \m^2)+ (\underline{x}) \subset (\underline{x} : \m^2),$$
which implies that $z_i \notin (\underline{x}: \m)$, in particular $z_{t+1} \notin (\underline{x}: \m)$. Therefore we 
can choose a minimal generator $z'$ of $\m$ for which $z' z_{t+1} \notin (\underline{x})$. 

But then there exists $\alpha \in K$ such that an element of the form
$z+ \alpha z'$ is a minimal generator for $\m$ and $(z+ \alpha z') z_{i} \notin (\underline{x})$ for all $i =1,..., t+1$:
indeed, this follows immediately for $i = t+1$ due to our choices, as for $i \leq t, zz_i$ is a multiple of the socle generator of $\underline{x} : \m$ and since $K$ is infinte we can choose
$\alpha$ such that $\alpha z' z_i$ is not the same multiple of the socle generator. 

Therefore, we contradicted the maximality of $t$. Hence we must have $t=k$.

}

\end{proof}

From now on, let us fix $z_1,\ldots, z_k \in R$ chosen as in Lemma~\ref{choice}.

Thus, modulo $(\underline{x})$, each $zz_i$, $i =1,
\cdots, k$ generates
the socle of $R/(\underline{x})$.

Let us denote $J_i = (z_i, \cdots, z_k, \underline{x})$, for all $i =
1, \cdots, k$.

Let $u$ in $R$ be an element
that generates the socle of $R/(\underline{x})$. Denote
  $J = (\underline{x}, u)$. Note that according to our
  remark on the elements $zz_i$, $J = (I, zz_i)$  for $i =1, \ldots, k$.

Denote $L_i = (\underline{x}, z_i)$ and $B_i = (\underline{x}) : L_i$.
Note that $L_k = J_k$.  Since $z_i \in (\ux):\m^2 - (\ux):\m$, the
chain $(\ux,z_i) \supsetneq (\ux, u) \supsetneq (\ux)$ is saturated,
i.e., $\length(L_i/(\ux)) = 2$.  So by duality, $\length(R/B_i) = 2$.
Since $zz_i \notin (\ux)$, the chain $R \supsetneq (z,B_i)=\m \supsetneq B_i$
is saturated.

For any $ q = p^e$, let $G_q = (\underline{x}\brq : \fm \brq)$. Note
that $J \brq \subset G_q$.

Consider a radical extension for $R$ and $z$, $S = R[v]$ such that 
$v^n = z$. Since $R$ is Henselian and $z \in \m$, $S$ is local.
Set $b = [\ff(S):\ff(R)] (\le n)$.
 Denote $\ehk(R) = 1 + \epsilon_R$, $\ehk(S) = 1 +
\epsilon_S$.

\medskip
\noindent
In what follows we will make a sequence of claims that will lead to our
main result.

\medskip
{\bf Claim (1):} $\ehk(B_i, \fm) \leq \frac{n}{n-1} \ehk(R) -
\frac{1}{b(n-1)}\ehk(S)$.

From our observations above about $R/B_i$ we can apply Lemma~\ref{bound} with
$I = B_i$ and $J = \m$
 to get $\ehk(B_i, \fm) \leq \frac{n}{n-1} \ehk(R) - \frac{n}{b(n-1)}\ehk(S)$.

\medskip
{\bf Claim (2):} $$\lim_{q \to \infty} \frac{1}{q^d} \length (G_q/J\brq) =
\ehk(R) -\ehk((\underline{x}), J).$$

We observe that $R/ (\underline{x}) \brq$ is Gorenstein Artinian.

So, by duality, 
$\length \left(\dfrac R{(\ux)\brq}\right) 
=\length\left(\Hom_R \left( \dfrac R{L\brq},
                        \dfrac R{(\underline{x})\brq}\right)\right)=
\length \left(\dfrac{(\ux)\brq : L\brq}{ (\ux)\brq}\right) $, for any
$\fm$-primary ideal $L$ in $R$. 

Let $L = \fm$ and we obtain $\length (G_q/(\underline{x})\brq) = \length
(R/ \fm \brq)$, so $\length (R/G_q) = \length (R/(\underline{x})\brq))
- \length (R/ \fm \brq)$, which is the same as

$$\length (G_q /J \brq) = \length (R/ \fm \brq) - \left(\length (R /
(\underline{x})\brq) - \length(R / J \brq)\right).$$ 
Dividing by
$q^d$, and taking the limit as $q \to \infty$ gives the claim.

\medskip
{\bf Claim (3):} $$\length (G_q/J\brq) \geq \length (\frac{\sum_{i=1}^k (L_i
  \brq \cap G_q)}{J \brq})$$

This is immediate since $\sum_{i=1}^k (L_i
  \brq \cap G_q) \subset G_q$.

Now, we need to introduce further notation:

for $i = 1, \ldots, k-1$, we let $$N_{i,q} = \frac{L_i \brq \cap G_q}{J
  \brq},$$ and put

$$ a_i : = {\rm limsup} \frac{1}{q^d} \length (\frac{ (L_i \brq \cap
  G_q) \cap \sum_{j =i+1}^k (L_j \brq \cap G_q)}{J \brq}), $$ so

$$ a_i = {\rm limsup} \frac{1}{q^d} \length (N_{i,q} \cap
  \sum_{j=i+1}^k N_{j,q}).$$
 
We set $a_k =0$.

\medskip
{\bf Claim (4):}   For any $i_0 \in \{1,\ldots,k-1\}$

$$\length (\sum_{i=i_0}^k N_{i,q}) = \sum_{i=i_0}^k \length
(N_{i,q}) -\sum_{i=i_0}^{k-1} \length (N_{i,q} \cap
  \sum_{j=i+1}^{k} N_{j,q}).$$

Write the following exact sequence

$$ 0 \to N_{i,q} \cap \sum_{j =i+1}^k N_{j,q} \to N_{i,q} \oplus \sum_{j
  =i+1}^k N_{j,q} \to \sum_{j =i}^k N_{j,q} \to 0$$

and now start with $i=i_0$ and recursively one gets the claim.

\medskip
{\bf Claim (5):} $$\length (N_{i,q}) \geq \length (\frac{L_i \brq}{J \brq}) -
\length (\frac{\fm \brq}{B_i \brq}).$$

From the short exact sequence

$$ 0 \to N_{i,q} \to \frac{L_i \brq}{J \brq} \to \frac{L_i
  \brq}{L_i\brq \cap G_q} \to 0$$

we see that $\length (\frac{L_i \brq}{J \brq}) = \length (N_{i,q}) +
\length (\frac{L_i \brq}{L_i\brq \cap G_q})$.

But $$\length (\frac{L_i \brq}{L_i\brq \cap G_q} )= \length( \frac{L_i \brq + G_q}{G_q}) \leq
  \length(\frac{(\underline{x}) \brq : B_i \brq}{G_q}) =\length( \frac{\fm \brq}{B_i \brq}).$$.

Hence $$\length (N_{i,q}) = \length (\frac{L_i \brq}{J \brq}) -
\length (\frac{L_i \brq}{L_i\brq \cap G_q}) \geq \length (\frac{L_i \brq}{J \brq}) -
\length (\frac{\fm \brq}{B_i \brq}).$$

\medskip

{\bf Claim (6):} $$\length (\frac{L_i \brq}{J \brq}) = \length (\frac{J_i
  \brq}{J_{i+1} \brq}) + \length (\frac{L_i \brq \cap J_{i+1} \brq}{J
  \brq}).$$

For all $i = 1, \ldots, k-1$, 

$L_i \brq + J_{i+1} \brq = J_i \brq$, so

$$ \frac{L_i \brq}{J \brq} / \frac{L_i \brq \cap J_{i+1} \brq}{J\brq}
\simeq \frac{J_i \brq}{J_{i+1} \brq}$$

and this gives the claim.

\medskip
{\bf Claim (7):} $$\length (\frac{ L_i \brq \cap G_q \cap (\sum_{j =i+1}^k
  L_j \brq \cap G_q)}{J \brq}) \leq \length (\frac{L_i \brq \cap
  J_{i+1} \brq}{J \brq}).$$

This follows immediately as  $L_i \brq \cap G_q \cap (\sum_{j =i+1}^k
  L_j \brq \cap G_q) \subset L_i \brq \cap J_{i+1} \brq$, since
$L_j \subseteq J_{i+1}$ for all $j \geq i+1$.

\begin{Thm}
\label{ineq3}
Let $(R, \fm)$ be a local Gorenstein ring. Let
$\underline{x}$ be a minimal reduction generated by a system of
parameters and let $z \in \fm \setminus (\underline{x})$ be 
a minimal generator of $\m$ picked as described above.

Let $S=R[v]$ be a radical extension for $R$ and $z$ of degree $n$. 
Let $b = [ff(S):ff(R)]$.  Then
\[
\ehk(R) \ge 
\begin{cases}
\dfrac{\e(n-1)}{\e n-2} + \dfrac{n(\e-2)}{b(\e n-2)} \ehk(S) & \text{if $k=\e-2$;}\\
\dfrac{\e(n-1)}{(n-1)\e + k +1}+ \dfrac{n(k+1)}{b\big((n-1)\e+k+1\big)}\ehk(S)
&\text{if $k < \e-2$.}
\end{cases}
\]
\end{Thm}

For $n = b = 2$, the first case gives $\ehk(R) \ge \dfrac{\e}{2(\e-1)} + 
\dfrac{\e-2}{2(\e-1)} \ehk(S)$ and the second case gives
$\ehk(R) \ge \dfrac \e{\e+k+1} + \dfrac{k+1}{\e+k+1} \ehk(S)$.

\begin{proof}
We will keep the notation introduced above and make references to
the claims just proved.

We see that $\length (\frac{G_q}{J \brq} ) \geq \length (\sum_{j=1}^k
  N_{j,q})$ and by Claim (4) and (5) we get

\begin{align*}
\length \left(\dfrac{G_q}{J \brq} \right) \geq \sum_{i=1}^k \length
(N_{i,q}) -\sum_{i=1}^{k-1} \length (N_{i,q} \cap
  & \sum_{j=i+1}^{k} N_{j,q}) 
\\
& \geq \sum_{i=1}^k \left( \length \left(\dfrac{L_i \brq}{J \brq}\right) -
\length \left(\dfrac{\fm \brq}{B_i \brq}\right)\right) 
-\sum_{i=1}^{k-1} \length (N_{i,q} \cap
  \sum_{j=i+1}^{k} N_{j,q}),
\end{align*}

 which by Claim (7) leads to 

$\length \left(\dfrac{G_q}{J \brq} \right) 
\geq \sum_{i=1}^k \length \left(\dfrac{L_i
    \brq}{J \brq}\right) - \sum_{i=1}^{k-1} \length \left(\dfrac{L_i \brq \cap
  J_{i+1} \brq}{J \brq}\right) 
- \sum_{i=1}^k \length \left(\dfrac{\fm \brq}{B_i \brq}\right)$

and now using Claim (6) this last term can be bounded below  by 

$$
\sum_{i=1}^{k-1} \length \left(\frac{J_i \brq}{J_{i+1} \brq}\right) -
\sum_{i=1}^k \length \left(\frac{\fm \brq}{B_i \brq}\right) 
+\length \left(\frac{L_k \brq}{J \brq}\right).
$$

But $L_k = J_k$, so we get

$$
\length \left(\frac{G_q}{J \brq} \right) 
\geq \sum_{i=1}^{k-1} \length \left(\frac{J_i \brq}{J_{i+1} \brq}\right) -
\sum_{i=1}^k \length \left(\frac{\fm \brq}{B_i \brq}\right) 
+ \length \left(\frac{J_k \brq}{J \brq}\right).
$$

Dividing by $q^d$ and taking the limits leads to

$$
\frac{1}{q^d} \lim_{q \to \infty} \length \left(\frac{G_q}{J \brq} \right) 
\geq \sum_{i=1}^{k-1} \ehk (J_{i+1} , J_i ) -
\sum_{i=1}^k \ehk(B_i, \fm) + \ehk(J, J_k).
$$

Consider the filtration

$$ 
(\underline{x}) \subseteq J \subseteq J_k \subseteq \cdots \subseteq J_2 \subseteq
J_1 \subseteq \fm.
$$

So, $\ehk ((\underline{x}))- \ehk (R)
= \ehk ((\underline{x}), J) + \ehk(J, J_k) + 
\sum_{i=1}^{k-1} \ehk(J_{i+1}, J_i)  + \ehk (J_1, \fm)$.

We have that $\ehk ((\underline{x})) = \e$ and $\lim_{q \to \infty} \frac{1}{q^d} \length (G_q/J\brq) =
\ehk(R) -\ehk((\underline{x}), J)$ as shown in Claim (2),

So, $\e -2\ehk(R) + \lim_{q \to \infty} \frac{1}{q^d} \length (G_q/J\brq) 
=\e -2\ehk (R) +\ehk(R) - \ehk((\underline{x}), J) =  
\ehk(J, J_k) + \sum_{i=1}^{k-1} \ehk(J_{i+1}, J_i) + \ehk (J_1, \fm)$.

But, $$\frac{1}{q^d} \lim_{q \to \infty} \length (\frac{G_q}{J \brq} ) 
\geq \sum_{i=1}^{k-1} \ehk (J_{i+1} , J_i ) -
\sum_{i=1}^k \ehk(B_i, \fm) +\ehk(J, J_k),$$ which says that

$\e -2\ehk(R) + \sum_{i=1}^{k-1} \ehk (J_{i+1} , J_i ) -
\sum_{i=1}^k \ehk(B_i, \fm) + \ehk(J, J_k) 
\leq \ehk(J, J_k) + \sum_{i=1}^{k-1} \ehk(J_{i+1}, J_i) + \ehk (J_1, \fm).$

By cancelling out the common terms, we see that

$\e \leq  \sum_{i=1}^k \ehk(B_i, \fm) +\ehk(J_1, \fm) + 2\ehk(R)$.

But $\ehk(J_1, \fm ) = \ehk (J_1) -\ehk (R)$.

We have also proved earlier that 
$ \ehk(B_i, \fm) \leq \frac{n}{n-1} \ehk(R) - \frac{n}{b(n-1)}\ehk(S)$.

So, 
$$ 
\e \leq k( \frac{n}{n-1} \ehk(R) - \frac{n}{b(n-1)}
\ehk(S)) + \ehk(J_1) + \ehk(R),
$$
which can be rearranged as

$$ \e \leq k( \frac{n}{n-1} \ehk(R) - \frac{n}{b(n-1)}
\ehk(S)) + \ehk(J_1, \fm) + 2\ehk(R).$$

If $k = \e-2$, then $J_1 = \m$, so $ \ehk(J_1, \fm) = 0$.  A small
amount of algebra gives the desired conclusion.

Assume that $k < \e-2$.  Then according to the set up for this case, we have that $J_1 \subsetneq \m$, $z \notin J_1$, and $z$ is a part of a minimal generating set for $\m$. Call this generating set $z, y_2, \ldots, y_h$ with $h=k+d$. Then $\m = (z, y_2, \ldots, y_h) +\m^2$.

So we may pick an
ideal $J_0= (y_2, \dots, y_h) + \m^2$ such that $J_1 \inc J_0 \inc J_0+(z) = \m$, where
$\length(\m/J_0) = 1$.  By Lemma~\ref{bound}, $\ehk(J_0,\m) \le 
\dfrac{n}{n-1} \ehk(R) - \dfrac{n}{b(n-1)}\ehk(S)$.  Also,
$\length(J_0/J_1) = \e-k-3$, so $\ehk(J_1,J_0) \le (\e-k-3)\ehk(R)$.
Putting this information into our inequality now yields
$$
\e \leq (k+1)( \frac{n}{n-1} \ehk(R) - \frac{n}{b(n-1)}
\ehk(S)) + (e-k-1)\ehk(R),
$$
and some algebra yields our other case.
\end{proof}

\subsection*{Lower bounds for the Hilbert-Kunz multiplicity of a Gorenstein F-regular ring:}

\medskip

We now begin a construction that will yield a lower bound for Gorenstein,
F-regular, non-regular local rings.  So assume that $(R,\m)$ is an
F-regular local ring of multiplicity $\e = \e(R) > 1$ and characteristic $p >2$.  By the results
in Section 4 we may actually assume that $\e\ge 6$.  Note that
$R$ must be a normal domain.  We may complete and 
extend the residue field to assume that it is algebraically closed.
Let $d = \dim R$ and $k = \mu(\m)-d$.
Let $\bx = x_1,\ldots, x_d$ be a minimal reduction of $\m$, so that
$\length(R/(\bx)) = \e$.  We now inductively choose $w_1,\ldots, w_d \in \m$
such that for each $i = 1,\ldots, d$, the set $w_1,\ldots, w_i, x_{i+1},
\ldots, x_d$ is a minimal reduction for $\m$, there is a set $A_i$ of minimal generators
of $(w_1,\ldots, w_i, x_{i+1},\ldots, x_d):\m^2$ (modulo 
$(w_1,\ldots, w_i, x_{i+1},\ldots, x_d)$) such that $w_{i+1}z \notin
(w_1,\ldots, w_i, x_{i+1},\ldots, x_d)$ for $z \in A_i$, if $k<\e-2$,
$w_{i+1} \notin (w_1,\ldots, w_i, x_{i+1},\ldots, x_d):\m^2$, and if $k = \e-2$,
$w_{i+1}$ belongs to $A_i$.  
Such a choice is due to our Lemma~\ref{choice}.

For convenience we let $\bw_i = w_1,\ldots, w_i$ and $\bx_{i+1} = x_{i+1},
\ldots, x_d$.

Now, fix $n$, and let $v_i = w_i^{1/n}$ be an $n$th root in $R^+$ for
$1 \le i \le n$.  As above, let $\bv_i = v_1,\ldots, v_i$.
Set $R_0 = R$ and for $i \ge 1$,
$R_i = R[v_1,\ldots, v_i] = R_{i-1}[v_i]$.
Each ring is henselian, so adjoining $v_i$ yields another local ring.
Moreover, all the residue fields are the same.
If we assume that $R_i$ is normal (e.g., if $R_i$ is F-regular),
then $R_{i+1} \cong R_i[X]/(X^n - w_{i+1})$, so $R_{i+1}$ is free of
rank $n$ over $R_i$  (since $R_i$ is normal, the minimal polynomial
of $v_{i+1}$ over $\ff(R_i)$ has coefficients in $R_i$, hence divides
$X^n - w_{i+1}$.  If it properly divides, then an interpretation of
the product of the constant terms involved will give $w_{i+1} \in
(\bw_i) + \m^2 \inc R$, meaning $w_{i+1}$ is not a minimal
generator of $\m$).  Thus, in the context of Theorem~\ref{ineq3}, applied
to $R_i \to R_{i+1}$, we have $n = b = [\ff(R_{i+1}),\ff(R_i)]$.

Let $t = \max\{i \mid R_i \textrm{\ is normal}\}$.  For $1 \le i \le t$ let
$\phi_i: \dfrac{R_0}{(\bw_i, \bx_{i+1})}
\twoheadrightarrow \dfrac{R_i}{(\bv_i, \bx_{i+1})}$.
We have that each $\phi_i$ is an
isomorphism. In particular, $\e(R_i) =\e$, for all $i \leq t$; also for $i \leq t$, $R_i$ is Gorenstein.

If we now write $\m_{R_0} = (\bw_i,\bx_{i+1})+J_i$
where $\mu(J_i) = \mu(\fm_{R_0})-d$ and $w_{i+1}$ is a minimal
generator of $J_i$, we have 
$(\bw_i,\bx_{i+1}):\m_{R_0}^2 =
(\bw_i,\bx_{i+1}):J_i^2$.
Note that $\m_{R_i} = (\bv_i,\bx_{i+1})+J_i$ (minimally).
The isomorphism $\phi_i$ now gives that 
\begin{align*}
\dfrac{(\bv_i,\bx_{i+1}):_{R_i}\m_{R_i}^2}{(\bv_i,\bx_{i+1})}
= \dfrac{\left((\bw_i,\bx_{i+1}):_{R_0}J_i^2\right) R_i
 + (\bv_i,\bx_{i+1})}{(\bv_i,\bx_{i+1})} \\
&=\dfrac{\left((\bw_i,\bx_{i+1}):_{R_0}\m_{R_0}^2\right) R_i
 + (\bv_i,\bx_{i+1})}
{(\bv_i,\bx_{i+1})}
\end{align*}

Since $ R'_0=\dfrac{R_0}{(\bw_i, \bx_{i+1})} \to R'_i=\dfrac{R_i}{(\bv_i, \bx_{i+1})}$ is an isomorphism of $R_0$-algebras we note that because the images of $J_i$ are minimal generators in the domain, they must be minimal generators in the codomain as well.  Moreover, $\Ann_{R'_0}(\fm^2_{0}) $ maps to $\Ann_{R'_i}(\fm^2_{i}) $ under the mentioned isomorphism, and so the minimal set of generators $A_i$ is a set of generators for
$$\dfrac{(\bv_i,\bx_{i+1}):_{R_i}\m_{R_i}^2}{(\bv_i,\bx_{i+1})},$$
and $w_{i+1} z \notin (\bv_i, \bx_{i+1})$ for $z \in A_i$ because $(\bv_i, \bx_{i+1}) \cap R_0= (\bw_i, \bx_{i+1})$.

Moreover, $\bw_i, \bx_{i+1}$ for a minimal reduction for $\fm_{R_0}$ hence $\bv_i, \bx_{i+1}$ form a minimal reduction for $\fm_{R_i}$. We also need that $v_1, \ldots, v_i, w_{i+1}, x_{i+2}, \ldots, x_d$ form a minimal reduction
of $\fm_{R_i}$.

When $k < \e-2$, $w_{i+1} \notin (\bw_i, \bx_{i+1}): \fm_{R_0}^2$. Since $\Ann_{R'_0}(\fm^2_{0}) $ maps to $\Ann_{R'_i}(\fm^2_{i}) $ under the isomorphism $R'_0 \to R'_i$ we get that $w_{i+1} \notin (\bv_i, \bx_{i+1}): \fm_{R_i}^2$. Finally if $k =\e-2$, then $w_{i+1} \in A_i$ by our initial choice.

This shows that Theorem~\ref{ineq3} may be applied to the extension
$R_i \to R_{i+1}$ if $R_i$ is F-regular, i.e., that $w_{i+1}$ satisfies the necessary conditions
to be chosen as the $z$ in Theorem~\ref{ineq3}.

We make several observations about  the case that we may obtain an $R_d$ in
the above manner.  If we write $\m_{R_0} = (w_1,\ldots, w_d) + J$ with
$\mu(J) = \mu(\m_{R_0}) -d$, then $\m_{R_d} = (v_1,\ldots, v_d) + J$.
Thus every generator of $J$ is in $\overline{(w_1,\ldots, w_d)R_{d}}
= \overline{(v_1^{n},\ldots, v_d^n)R_d} = \overline{\m^n_{R_d}}$.
In addition, we note that via the isomorphism $\phi_d$ we may
filter $R_d/(v_1,\ldots, v_d))$, by essentially the same filtration
as we take of $R_0/(w_1,\ldots, w_d)$.  Let 
$r = \max\{ j\mid (\m_{R_0}^j + (w_1,\ldots, w_d))/(w_1,\ldots, w_d) \ne 0\}$.
We may then take a socle generator $u \in \m_{R_0}^r$, modulo $(\bw_d)$.
The same element will now represent a socle element in $R_d/(\bv_d)$,
and will have valuation at least $rn$.  Hence, if $rn\ge d$ then by
the Brian\c con-Skoda Theorem, $u \in (\bv_d)^*$, and $R_d$ is not
F-regular.

In particular, if $n \ge \lceil d/2\rceil$ (if $k = \e-2$), or $n \ge
\lceil d/3 \rceil$ (if $k < \e-2$, by Lemma~\ref{prep}), the ring $R_d$ cannot be F-regular.

Choose such $n$ and let $s = \max\{i : R_i \textrm{\ is  F-regular}\}$.  
Note $s<d$, and hence $R_{s+1}$ is not F-regular.

In each application of the theorem, $b = n$, so we obtain from 
Theorem~\ref{ineq3}, for each $i \le t$ (or $i < d$ if $t=d$) that
\[ 
\ehk(R_{i}) \ge
\begin{cases}
1 + \dfrac{\e-2}{\e n -2} (\ehk(R_{i+1})-1) & \textrm{ if $k = \e -2$;}\\
1 +  \dfrac{k+1}{(n-1)\e + k + 1} (\ehk(R_{i+1})-1) 
&\textrm{  if $k < \e-2$.}
\end{cases}
\]
  
By Corollary 3.10~\cite{AE}, $\ehk(R_{s+1}) \ge 1 +1/d$.

Hence 
\[ 
\ehk(R_{0}) \ge
\begin{cases}
1 + \left(\dfrac{\e-2}{\e n -2}\right)^{s+1} \left(\dfrac {\e}2 - 1 \right) 
& \textrm{ if $k = \e -2$;}\\
1 +\left(\dfrac{k+1}{(n-1)\e+k+1}\right)^{s+1}\left( \dfrac{1}{d}
\right) 
&\textrm{  if $k < \e-2$.}
\end{cases}
\]

We then get the following lower bounds for non-regular rings,
using that we may assume that  $6 \le \e \le d!$, $k \ge 3$: 
\[ 
\ehk(R_{0}) \ge
\begin{cases}
1 + \left(\dfrac{4}{6\lceil d/2\rceil-2}\right)^{d}\cdot 2 
& \textrm{ if $k = \e -2$;}\\
1+\left(\dfrac{4}{(\lceil d/3\rceil)d! +4}\right)^d\left( \dfrac{1}{d}\right) 
&\textrm{  if $k < \e-2$.}
\end{cases}
\]

Therefore we can state the final result:

\begin{Thm}
Let $R$ be a local Gorenstein F-regular ring of dimension $d \geq 2$ and Hilbert-Samuel multiplicity $\e \geq 6$ and positive characteristic $p >2$. Let $k = \embdim(R)-\dim(R)$. Assume further that $R$ is not a complete intersection.

Then if $\e \geq d! +1$ then $\ehk(R) \geq 1 + \frac{1}{d!}$. Otherwise

\[ 
\ehk(R) \ge
\begin{cases}
1 + \left(\dfrac{4}{6\lceil d/2\rceil-2}\right)^{d}\cdot 2
& \textrm{ if $k = \e -2$;}\\
1+\left(\dfrac{4}{(\lceil d/3\rceil)d! +4}\right)^d\left(\dfrac{1}{d}\right) 
&\textrm{  if $k < \e-2$.}
\end{cases}
\]
\end{Thm}

\begin{proof}
It suffices to remind the reader that the first claim is well-known (see~\cite{BE}). The last inequality is what we have proved in Section 6.
\end{proof}

\end{document}